\documentclass{amsart}
\usepackage{amsfonts}
\usepackage[alphabetic]{amsrefs} 

\newtheorem{theorem}{Theorem}[section]
\newtheorem{lemma}[theorem]{Lemma}
\theoremstyle{definition}

\newtheorem{example}[theorem]{Example}

\theoremstyle{remark}
\newtheorem{remark}[theorem]{Remark}
\numberwithin{equation}{section}
\theoremstyle{plain}

\newtheorem{proposition}[theorem]{Proposition}

\newcommand{\gp}{\dot{\gamma}}

\begin{document}
\title {Sectional curvature for Riemannian manifolds with density}

\author{William Wylie}
\address{215 Carnegie Building\\
Dept. of Math, Syracuse University\\
Syracuse, NY, 13244.}
\email{wwylie@syr.edu}
\urladdr{https://wwylie.expressions.syr.edu}
\thanks{The author was supported in part by NSF-DMS grant 0905527. }

\begin{abstract}
In this paper we introduce two new notions of sectional curvature for Riemannian manifolds with density.   Under both notions of curvature  we  classify the constant curvature manifolds.   We  also prove generalizations of  the theorems of Cartan-Hadamard, Synge, and Bonnet-Myers as well as a  generalization of the (non-smooth) 1/4-pinched sphere theorem.  The main idea  is to modify  the radial curvature equation and second variation formula and then apply the techniques of classical Riemannian geometry to these new equations.      \end{abstract}

\maketitle

\section{Introduction}

In this paper we are interested in the geometry of a Riemannian manifold $(M,g)$ with a smooth positive density  function, $e^{-f}$.  A theory of Ricci curvature for these spaces goes back to Lichnerowicz \cite{Lich1, Lich2} and was later developed by Bakry-Emery \cite{BE} and many others. It has turned out to be integral to developments in both  Ricci flow and optimal transport and has thus experienced an explosion of results in the last few years.  We will not try to reference them all  here, see  chapter 18  of \cite{MorganBook} for a partial survey.  A notion of weighted scalar curvature also comes up in Perelman's work \cite{Per} and is related to his  functionals  for the Ricci flow, also see  \cite{Lott2, CGY1, CGY2}.  The weighted Gauss curvature and the weighted Gauss Bonnet theorem in dimension two has also been studied in  \cite{ Cetc, CM}. 

We introduce two new concepts of  sectional curvature for a Riemannian manifold equipped with a smooth vector field $X$.   Given an orthonormal pair of vectors $(V,U)$ we define 
\begin{eqnarray*}
\sec_X^V(U) &=& \sec(V,U) + \frac12 L_X g(V,V) \\
\overline{\sec}_X^V(U) &=& \sec(V,U) + \frac12 L_X g(V,V) + g(X,V)^2 
\end{eqnarray*}
Where $\sec(V,U)$ is the sectional curvature of the plane spanned by $V$ and $U$.  When $X = \nabla f$ is a gradient field we write $\sec_f$ and $\overline{\sec}_f$ respectively.  The asymmetrical placement of $U$ and $V$  emphasizes that   $\mathrm{sec}^V_X(U) \neq \mathrm{sec}^U_X(V)$.   On the other hand,  we will  see below that  in dimension $2$, bounds on $\sec_X$ and $\overline{\sec}_X$ are equivalent to bounds on certain Bakry-Emery Ricci tensors.        Since the Bakry-Emery Ricci tensors generically have distinct eigenvalues in dimension $2$, the lack of symmetry of the weighted sectional curvature is a necessary  feature of any notion of weighted sectional curvature that agrees with the Bakry-Emery Ricci curvature.   We  also show below that $\sec_X$ and $\overline{\sec}_X$ come up naturally in  at least three places: the radial curvature equation, the second variation of energy formula, and formulae for Killing fields.   We will discuss our motivation for considering these notions from the radial curvature equation in section two.  

We use these equations to show that some of the fundamental comparison results about sectional curvature bounds extend to $\sec_X$ and $\overline{\sec}_X$.    We  define the condition $\mathrm{sec}_X = \psi$ where $\psi$ is a real valued function on $M$  to mean that  $\sec^V_X(U) = \psi(p) $ for all $p \in M$ and for all orthonormal pairs of $(V,U)$ in $T_pM$.   We can define the conditions  $\overline{\sec}_X = \psi$ or  $\sec_X \geq (\leq ) \psi$, etc. similarly.     

The most fundamental fact  about sectional curvature is that constant curvature characterizes the classical Euclidean, spherical, and hyperbolic geometries.  Constant weighted sectional curvature also characterizes natural vector fields or functions on  spaces of constant curvature. 

 \begin{proposition}[Constant Curvature] \label{ConsCurvIntro}
Let $(M^n,g)$  be a Riemannian manifold of dimension $n>2$, let $X$ be a smooth vector field and $f$ be a smooth function  on $M$, then 
\begin{enumerate}
 \item $\sec_X = \psi$   if and only if   $g$ has constant curvature and $X$ is a conformal field on $(M,g)$. 
 \item $\overline{\mathrm{sec}}_f = \psi$ if and only if   both $g$ and $e^{-2f} g$  have constant curvature. 
 \end{enumerate}
 \end{proposition}

When the weighted sectional curvatures are not constant, we think of $\overline{\sec}_X$ or $\sec_f$ as measuring how far away the space is from one of the canonical spaces in Proposition \ref{ConsCurvIntro}.    First we generalize  the Cartan-Hadamard theorem to the case where  $\overline{\sec}_X \leq 0$. 
 
 \begin{theorem}[Weighted Cartan-Hadamard Theorem] If  a complete  Riemannian manifold admits a smooth vector field  $X$  such that $ \overline{\sec}_X \leq 0$,
 then $M$ does not have any conjugate points.  In particular,  if $M$ is simply connected then it is diffeomorphic to $\mathbb{R}^n$. 
 \end{theorem}
 
Applying the result to the universal cover gives the standard corollary  that  a compact Riemannian manifold that admits a vector field $X$ with $\overline{\sec}_X \leq 0$ must have infinite fundamental group and have all other homotopy groups vanish.  The lack of conjugate points also implies much more about the fundamental group, see \cite{CS}.  Also  note that  $\overline{\mathrm{sec}}_X \geq \sec_X$, so the condition $\overline{\sec}_X \leq 0$ is a stronger assumption that $\sec_X \leq 0$. The Cartan-Hadamard theorem is not true for the condition $\sec_f \leq 0$, see Example \ref{cigar}.

 In the case of positive curvature we also prove generalizations of the following  results of Synge\cite{Synge} and Berger\cite{BergerKilling}.
 \begin{theorem}  \label{PosIntro} Suppose a compact  Riemannian manifold admits a smooth function  $f$ such that $\overline{\sec}_f > 0$  then 
 \begin{enumerate}
 \item If $M$ is even dimensional then every Killing field has a zero. 
 \item If $M$ is even dimensional and orientable, then $M$ is simply connected. 
 \item If $M$ is odd-dimensional, then $M$ is orientable.
 \end{enumerate}
 \end{theorem}
 
 \begin{remark}  The conditions $\overline{\sec}_f >0$ or $\sec_X >0$ for a compact manifold, has been studied much further by the author and Kennard in \cite{KennardWylie}. \end{remark}

In the case of a two sided bound on curvature, we also prove the following  generalization of the homeomorphic $1/4$-pinched sphere  theorem.  Our generalization will depend on the maximum and minimum of $u=e^f$, which we denote by $u_{max}$ and $u_{min}$.
 
 \begin{theorem} \label{PinchingIntro}
If $(M,g)$ is compact, simply connected Riemmanian manifold and there is a smooth function $f$  such that 
\[  \frac{1}{4}\left (\frac{u_{max}}{u_{min}} \right)^2 <  \overline{\sec}_f \leq  \left(\frac{u_{min}}{u_{max}} \right)^2,\] 
then $M$ is homeomorphic to the sphere. 
\end{theorem} 

The proof of Theorem \ref{PinchingIntro} follows from classical methods of Klingenberg \cite{Kling61} and Berger \cite{Berger2}.  We  prove that the manifold is homotopic to the sphere and apply the resolution of the Poincare conjecture to conclude the manifold is  homeomorphic to a sphere.   We do not know to what extent this theorem is optimal. Note that  the hypothesis implies that $\frac{u_{max}}{u_{min}} \leq (4)^{1/4} \approx 1.414$, so the result applies only to small densities.    Some other pinching phenomena will be  considered by the author in \cite{Wylie}. 
 
One  reason for studying sectional curvature for manifolds with density is that understanding sectional curvature will enhance our understanding of weighted Ricci curvature.    Given any real number $N$, the $N$-Bakry-Emery Ricci tensor is 
 \[ \mathrm{Ric}_X^N = \mathrm{Ric} + \frac12 L_X g - \frac{X^\sharp \otimes X^\sharp}{N} \]
When $N=\infty$ we write $\mathrm{Ric}^{\infty}_X = \mathrm{Ric}_X = \mathrm{Ric} + \frac12L_X g$.   As we mention above, comparison geometry for lower bounds on the Bakry-Emery Ricci tensors have been very well studied recently.  Traditionally this has been done with the parameter $N >0$ or infinite.  Recently, however the negative case has been considered, see \cite{KolMil, Mil, Ohta} and the references there-in.  Our approach to weighted sectional curvature  gives a new diameter estimate for a positive lower bound on $\mathrm{Ric}_f ^{-(n-1)}$.    
 
 \begin{theorem} \label{DiamIntro} If a complete  Riemannian manifold supports a bounded function $f$ such that $\mathrm{Ric}_f^{-(n-1)}   \geq (n-1) kg$
for some $k>0$,  then $M$ is compact with finite fundamental group and $ \mathrm{diam}_M \leq \left(\frac{u_{max}}{u_{min}} \right)^{\frac{1}{n-1}}\frac{\pi}{\sqrt{k}} .$
 \end{theorem}
  In \cite[Theorem 1.4]{WW} the author and Wei proved a similar  diameter bound   under  the stronger hypothesis of a positive lower bound on $\mathrm{Ric}_f$.  There are simple examples showing that $f$ being bounded is a necessary assumption for $M$ to be compact. Also see  \cite{Morgan2} for a similar result for the weighted diameter.

%
%
%

The paper is organized as follows.  In the next section we discuss the motivation for the definitions which  come from the  Bakry-Emery Ricci curvatures and the radial curvature equation.  We also discuss the relationship between our curvature and the curvature of the conformal change. In section 3 we discuss the case of constant weighted curvature;   in section 4 we discuss conjugate radius estimates;   in section 5 we consider  the second variation of energy  formula; in section 6 we prove the diameter estimate; and in section 7 we discuss curvature pinching.  In the final section we consider Killing Fields.


\textbf{Acknowledgement:} The author would like to thank Guofang Wei, Peter Petersen, and Frank Morgan  for their  encouragement and very  helpful discussions and suggestions that  improved the draft  of this paper.

\section{Motivation and the fundamental equations} 
In this section we first discuss the motivation for the Bakry-Emery Ricci curvature  in terms of Bochner formulas and  then show how a similar approach yields the definitions of $\mathrm{sec}_X$ and $\overline{\sec}_X$. 

\subsection{Ricci for manifolds with density}

Recall the Bochner formula for the Riemannian Laplacian 
\[ \frac12\Delta|\nabla u|^2 = |\mathrm{Hess} u|^2 + \mathrm{Ric}(\nabla u, \nabla u) + g\left( \nabla \Delta u, \nabla u\right)  \qquad u \in C^3(M). \]
If $\mathrm{Ric} \geq k$ and the dimension of $M$ is less than or equal to $n$ an application of the  Cauchy-Schwarz inequality gives
\begin{eqnarray}
\label{Bochner}  \frac12\Delta|\nabla u|^2 \geq \frac{(\Delta u)^2}{n} + k |\nabla u|^2 + g\left( \nabla \Delta u, \nabla u\right)  \qquad u \in C^3(M).  
\end{eqnarray}
For a smooth  vector field $X$   we consider the ``drift" Laplacian $\Delta_X= \Delta - D_X$.  A simple calculation gives the  Bochner formula 
\[ \frac12 \Delta_X |\nabla u|^2 = |\mathrm{Hess} u|^2 + \mathrm{Ric}(\nabla u, \nabla u) +\frac12L_Xg(\nabla u, \nabla u)  + g\left( \nabla \Delta_X u, \nabla u\right)  \qquad u \in C^3(M).\]   
The $N$-Bakry Emery Ricci tensor  is defined to be  $\mathrm{Ric}^N_X = \mathrm{Ric} + \frac12L_Xg- \frac{X^{\sharp} \otimes X^{\sharp}}{N}$.  When $N >0$, if $\mathrm{Ric}_X^N \geq k$ one can show that 
\begin{eqnarray*}
  \frac12\Delta_X|\nabla u|^2 \geq \frac{(\Delta u)^2}{n+N} + k |\nabla u|^2 + g\left( \nabla \Delta_X u, \nabla u\right).  
\end{eqnarray*}
This looks exactly like the Bochner formula for a $n+N$ dimensional manifold.  We can also consider the case where $N = \infty$, then we have the Bakry-Emery Ricci tensor $\mathrm{Ric}_f = \mathrm{Ric} + \mathrm{Hess} f$, and $\mathrm{Ric}_f \geq k$ gives 
\begin{eqnarray*}
  \frac12\Delta_X|\nabla u|^2 \geq  k |\nabla u|^2 + g\left( \nabla \Delta_Xu, \nabla u\right). 
\end{eqnarray*}
From these formulae one  can prove versions of many comparison  results for lower bounds on  $\mathrm{Ric}_X^N$ or  $\mathrm{Ric}_X$.  All of the classical results generalize to the  $\mathrm{Ric}_f^N$  case but with all of the  dimension dependent constants now depending  on the synthetic dimension $n+N$ (see \cite{Qian}).  We can think of $\mathrm{Ric}_f$ as being an infinite dimensional (or dimension-less) condition and thus the results for lower bounds on $\mathrm{Ric}_f$ are weaker, see for example \cite{Lott, Morgan, WW, MW}.     

%
  
\subsection{The radial curvature equation}
Now to consider sectional curvature  we  examine the special case of  the Bochner formula applied to a distance function.  Fix $p \in M$ and  let $r(x) = d(p,x)$.  The function $r$ is smooth on $M \setminus C_p$ where $C_p$ denotes the cut locus of $p$.  On $M \setminus C_p$  introduce geodesic polar coordinates  $(r, \theta)$ where $\theta \in S^{n-1}$. The Bochner formula  applied to the function $r$ then gives 
\begin{equation}  \label{Ric1}   \partial_r  (\Delta_X r )= -|\mathrm{Hess} r|^2 - \mathrm{Ric}_X (\partial_r, \partial_r).  \end{equation} 
In the case where $X = \nabla f$, the weighted Laplacian is also related to  the weighted volume by the equation   \begin{equation}  \label{Ric2} L_{ \partial_r} \left(e^{-f} d\mathrm{vol}\right)  =\left(\Delta_f r\right) e^{-f} d \mathrm{vol}. \end{equation}  Putting these two equations together we can then see how  bounds on Bakry-Emery Ricci tensors   gives  control on the  measure $e^{-f} d\mathrm{vol}_g$.

The corresponding equations for a distance function that involve sectional curvature are the fundamental equations.   
\begin{eqnarray}
\label{sec1} L_{\partial_r} g &=& 2\mathrm{Hess} r \\
\label{sec2} (\nabla_{\partial_r} \mathrm{Hess} r) (X,Y) + \mathrm{Hess}^2 r(X,Y) &=& -g( R^{\partial_r}(X), Y)
\end{eqnarray}
Where $\mathrm{Hess}^2 r$ is the operator square of $\mathrm{Hess} r$, namely if $S$ is a dual $(1,1)$-tensor to $\mathrm{Hess} r$,  $\mathrm{Hess} r (X,Y)= g(S(X), Y)$, then $\mathrm{Hess}^2 r = g(S(S(X), Y))$ and our notation for the curvature tensor is that 
\begin{eqnarray*}
R^V(U) =  R(U,V)V = \nabla_U \nabla_V V - \nabla_V \nabla_U V - \nabla_{[U,V]} V .  
\end{eqnarray*} So that $R^V$ is a symmetric operator on the orthogonal complement of $V,$ which, following \cite{Petersen} we call the  {\em directional curvature operator} in the direction of $V$. 

  Note that  if we trace equations  (\ref{sec1}) and (\ref{sec2}) we get equations (\ref{Ric1}) and (\ref{Ric2}). (\ref{sec2}) is called the radial curvature equation. 

For the moment we consider the gradient case,   $X = \nabla f$. The weighted  sectional curvatures  will control the growth of  $e^{-2f} g$ along a geodesic $\gamma$.   Consider the equation
\begin{eqnarray*}
L_{\partial_r} \left(e^{-2f} g\right) = 2e^{-2f} \left( \mathrm{Hess} r  - g(\nabla f, \partial_r) g \right).
\end{eqnarray*}
Set $H_f r=  \mathrm{Hess} r  - g(\nabla f, \partial_r) g $,  then
\begin{eqnarray*}
(\nabla_{\partial_r} (H_f r)) (X,Y)  &=& (\nabla_{\partial_r} \mathrm{Hess} r )(X,Y) - \mathrm{Hess} f (\partial_r, \partial_r) g(X,Y) \\
&=&   -  \mathrm{Hess}^2 r(X,Y)  - R^{\partial_r } _f(X,Y) \end{eqnarray*} 
Where $R^{V } _f(U,W) = g(R_f^{V}(U), W)$ is the weighted directional curvature operator defined as 
\begin{eqnarray*}
R^V_X(U)  &=&  R^V(U) + \frac12 L_X g(V,V) U   
\end{eqnarray*}
with $X = \nabla f$, so that if $(V,U)$ is an orthonormal pair of vectors,  $g(R_X^V(U), U) = \sec_X^V(U)$. 

We can make these equations more concrete by considering Jacobi fields.  For a Jacobi field $J$ along a unit speed radial geodesic, $\gamma(r)$,  with $J \perp \gp$ the fundamental equations are
\begin{eqnarray*}
\partial_r  |J|^2  &=& 2 \mathrm{Hess} r(J,J) \\
\partial_r  \left( \mathrm{Hess} r (J,J) \right) &=&  \mathrm{Hess}^2 r(J,J) - R(J, \partial_r, \partial_r, J).
\end{eqnarray*}
When considering Jacobi fields in the  weighted case,  the curvatures $\overline{\sec}_f$ appear.  Let  
\begin{eqnarray*}
 \overline{R}^V_X(U) &=&   R^V(U) + \left(\frac12L_X g(V,V) + g(X,V)^2 \right) U
 \end{eqnarray*}
 and $\overline{R}^{V} _X(U,W) =  g(\overline{R}^{V} _X(U), W)$,  and for $X=\nabla f$ write  $\overline{R}_f$
then we have 
 \begin{eqnarray*}
\partial_r \left(e^{-2f}  |J|^2 \right) &=& 2 e^{-2f} \left( H_f r (J,J) \right) \\
\partial_r \left( H_f r(J,J)  \right) 
&=&  \mathrm{Hess}^2 r(J,J)- 2g(\nabla r, X)g(\dot{J}, J)   - R^{\partial_r } _f(J,J)\\
&=& (H_fr)^2(J,J) - \overline{R}^{\partial_r } _f(J,J) .
\end{eqnarray*}
Which now looks even closer to the radial curvature equation (\ref{sec2}).  

Jacobi fields are the variation fields produced by  variations of geodesics. So we can think of Jacobi fields  as measuring the rate of the spreading of geodesics   and of the fundamental equations as showing that sectional curvature controls this spreading.  Thus, the weighted sectional curvatures control the rate of spreading of geodesics in a weighted sense  by controlling the derivative of  $e^{-2f}|J|^2$ along geodesics. 

In the motivation above we have used that $X = \nabla f$ in order to  differentiate $e^{-2f} g$.  However, many of the arguments we will use only depend on arguing along a fixed geodesic $\gamma$.  Along a  fixed geodesic $\gamma$ we can always find an anti-derivative for $X$ by simply defining $f_{\gamma}(t) = \int_0^t g(X, \gp) dt.$ 
We can then  still make sense of the equations above along $\gamma$, replacing $e^{-2f}$ with $e^{-2f_{\gamma}}$.

We have  first motivated the definition of the weighted sectional curvature through the radial curvature equation because it  is closer to the approach of  Bakry-Emery and Lichnerowicz in the Ricci curvature case. We  consider the second variation of energy formula in section 5.  The weighted curvature also comes up in considering equations for Killing fields, as we will show in section 8.

\subsection{Relationship to the conformal change}
The weighted curvatures are also  different from  the sectional curvatures of the conformal metric $h = e^{-2f}g$.  The formula for the $(4,0)$-curvature tensor of $h$ in terms of the curvature of $g$ is  
\begin{eqnarray*}
R^h = e^{-2f} \left( R^g + \left(  \mathrm{Hess}^g f  + df \otimes df - \frac12 |df|^2 g \right) \circ g\right) 
\end{eqnarray*}
Where $\circ$ denotes the Nomizu-Kulkarni product.    We can  re-interpreted this formula in the following way.   

\begin{proposition} \label{Conformal} 
Let $(M,g)$ be a  Riemannian manifold with $f$ a smooth function on $M$ and let  $h=e^{-2f} g$ then 
\[ \left(\overline{R}^g\right)_f^U(V,V) = e^{2f}  \left(\overline{R}^h\right)_{-f}^V(U,U).  \]
In particular, 
\[ \overline{\sec}^g_f(U,V) = e^{-2f} \left( \overline{\sec}^h_{-f}(V,U)\right). \]
\end{proposition}

\begin{remark}
This proposition shows  that the map $(g,f) \rightarrow (e^{-2f} g, -f)$ is an involution on the space of Riemannian metrics with density that  preserves the conditions  $\overline{\sec}_f = \phi$, $\overline{\sec}_f  \geq 0$ or $\overline{\sec}_f \leq 0$. 
\end{remark}

\begin{proof} 
Let $U,V$ be orthogonal vectors in $g$.  Then we have 
\begin{eqnarray*}
e^{2f} R^h(U,V,V, U) &=& R(U,V,V,U) + \mathrm{Hess}^g f(U,U) g(V,V) +  \mathrm{Hess}^g f(V,V) g(U,U) \\
&&  + df(U)^2g(V,V) + df(V)^2 g(U,U) - |df|^2g(U,U)g(V,V) 
\end{eqnarray*}
Which gives us 
\begin{eqnarray*}
\overline{R}^{U} _f(V,V) &=& e^{2f} \left( R^h(U,V,V, U) - \mathrm{Hess}^g f(V,V) h(U,U) - df(V)^2 h(U,U)+ |df|^2g(V,V) h(U,U)  \right)  \\
&=&  e^{2f} \left( R^h(U,V,V, U) - \mathrm{Hess}^h f(V,V) h(U,U) + df(V)^2 h(U,U) \right) \\
&=&  e^{2f} \left( \overline{R}^{V} _{-f} (U,U) \right). 
\end{eqnarray*}
Where we have used the formula for the Hessian under the change of metrics
\[ \mathrm{Hess}^h f (U,V) = \mathrm{Hess}^g f(U,V) + 2 df(U)df(V) - |df|^2 g \]
\end{proof}

%
%
%
%
 
\section{Constant Curvature}

  In this section we establish that  our definitions of  constant sectional curvature  characterize natural canonical  Riemannian manifolds with density  in dimension larger than two.
  
    First we consider the case $\sec_X = \psi$ for some function $\psi$.   In dimension two  we always have $\mathrm{sec} = \phi$ and so  $\sec_X = \psi$ if and only if $X$ is a conformal field.  An obvious example in higher dimensions  is a constant curvature metric with $X$ a conformal field.   It is, in fact easy to see from Schur's lemma that these are the only examples.   
  
  \begin{proposition} 
Suppose that $(M^n,g)$  has  $n>2$. There is  a vector field $X$ on $(M,g)$  such that $\mathrm{sec}_X = \psi$
for some function $\psi:M \rightarrow \mathbb{R}$ if and only if $(M,g)$ is a space of constant curvature and $X$ is a conformal field on $(M,g)$.   Moreover, if $\psi=K$ is constant then either $X$ is a Killing field or $(M,g)$ is isometric to a domain  of  Euclidean space and $X$ is a homothetic field satisfying $L_X g = K g$.
\end{proposition}

\begin{proof}
Let $U,V$ be perpendicular unit vectors in $T_pM$, then 
\begin{eqnarray*}
\psi  &=& \sec^U_X(V)  = \sec(U,V) + L_Xg(U,U)\\
\psi &=&  \sec^V_X(U) = \sec(V,U) + L_Xg(V,V)
\end{eqnarray*}
Since $ \sec(U,V) =   \sec(V,U)$, we have $L_Xg(U,U) = L_Xg(V,V)$, showing that $X$ is a conformal field, $L_Xg = \phi g$, $\phi:M \rightarrow \mathbb{R}$. Then, letting $\{E_i\}_{i=1}^{n-1}$ be an orthonormal basis for the orthogonal complement of $U$ we have 
\[ (n-1) \psi = \sum_{i=1}^{n-1} \sec^U_X(E_i)  = \mathrm{Ric}(U,U) + (n-1) \phi \]
So that $\mathrm{Ric} = (n-1)(\psi - \phi)g$.  By Schur's lemma,  $\psi - \phi$ must be  constant, showing the metric has constant curvature.

This also shows that $\psi=K$ is constant if and only if $\phi$ is.   If $\phi$ is zero, then $X$ is Killing and $(M,g)$ has constant curvature $K$.  If $\phi \neq 0$, then $X$ is a non-Killing homothetic field.  The existence of such a field implies that  $(M,g)$ is isometric to a domain in Euclidean space, see \cite[p. 242]{KN}. 
\end{proof}
In the case  $\overline{\mathrm{sec}}_X = \psi$ the same proof  gives  the following result. 

\begin{lemma} \label{ConstantB}
Suppose that $(M^n,g)$  has $n>2$ and  there is a vector field $X$ on $M$  such that $\overline{ \mathrm{sec}}_X = \psi$ for some function $\psi$
then $(M,g)$ has constant curvature $\rho$ and  $X$ satisfies $ L_X g + X^{\sharp} \otimes X^{\sharp} = (\psi-\rho) g. $
 \end{lemma}
When  $X = \nabla f$ this gives us the following statement from the introduction. 

\begin{proposition}
 $(M^n,g,f)$  has $n>2$ and $\overline{ \mathrm{sec}}_f = \psi$
if and only if $g$ and $h=e^{-2f} g$ have constant curvature.  
 \end{proposition}
 
 \begin{proof}
$g$ has constant curvature  by Lemma \ref{ConstantB} and from Proposition \ref{Conformal},  $\sec^h_{-f}  = \psi e^{-2f}$.   Therefore  applying Lemma \ref{ConstantB}  to $h$ tells us that  $h$ also has constant curvature.  Conversely, if $g$ and $h$ are both constant curvature, the equation for the curvature tensor under conformal change shows that $\mathrm{Hess} u$ is a function times the metric, which implies that $\overline{\sec}_f = \psi$. 
\end{proof}

The conformal changes between Riemannian metrics with constant curvature, are completely classified  in fact they are known in  the Einstein case, see \cite{Brinkmann, KR}.  The proof of this fact also gives some more  information about the possible  functions $\psi$ such that $\overline{\sec}_f = \psi$.     Letting $u=e^f$, from the equation in Lemma \ref{ConstantB} we have $\mathrm{Hess} u = (\psi-\rho) u g$ where $\rho$ is the curvature of the metric. A lemma of  Brinkmann-Tashiro  states that if one has a non-constant  solution to $\mathrm{Hess} u = \phi g $  for some function $\phi$, then the metric must be   of the form 
\[ g = dt^2 + (u'(t))^2 g_N\] where $u$ is a function of $t$ and $g_N$ is some fixed metric.  Brinkmann  \cite{Brinkmann} showed that  this is true locally and Tashiro \cite{Tashiro} showed it is true globally when  the metric is complete, also see \cite{OS, JW}.

Once we have these coordinates  we can compute that $ \mathrm{Hess} u = u'' g $
where prime denotes derivatives in the $t$ direction.  So we have that $u$ is a solution to $u'' = (\psi-\rho)u$.  Differentiating this equation gives us  $u'''= (\psi u)' - \rho u'$.  On the other hand the sectional curvature  in these coordinates  is given by $ \mathrm{sec}(\partial_r, X) = -\frac{u'''}{u'}.$
Since  $\rho$ is also  the sectional curvature these two equations combine to give us $(\psi u)' = 0$, i.e.  $\psi = Ku^{-1}$ for some constant $K$. 

In particular, we can see that if $\overline{\sec}_f = K$  for a constant $K$  and $f$ is non constant, then $K$ must be zero. In this case  we get the following classification in terms of the curvature $\rho$.    

\begin{example} Suppose that $(M^n,g,f)$  has $n>2$ and  $\overline{ \sec}_f = 0$. 
If $f$ is non-constant,  then after  normalizing $\rho$ to be 1, 0, or -1  and  possibly re-parametrizating $r$ and rescaling the metric $g_N$  below, the only possibilities are
\begin{enumerate}
\item $\rho=1$,  $g = dr^2 + \sin^2(r)  g_N $,  where $g_N$ is a metric of constant curvature 1,  and $u = \cos(r)$. 
\item $\rho=0$, $g= dr^2 + g_N$ where $g_N$ is a flat metric, and $u=Ar$.
\item $\rho=-1$  and either
\begin{enumerate}
\item $g = dr^2 + \sinh^2(r) g_N$ where $g_N$ is a metric of constant curvature 1,  and $u = \cosh(r)$. 
\item $g=dr^2 + e^{2r} g_N$ where $g_N$ is a flat metric, and $u = e^r$. 
\item $g=dr^2 + \cosh^2(r) g_N$ where $g_N$ is a metric of constant curvature $-1$, and $u=\sinh(r)$. 
\end{enumerate}
\end{enumerate}
  \end{example}
  
  
  \begin{remark}  \label{RemTri} These examples already show that there is no obvious  Toponogov triangle comparison type theorem for the conditions $\overline{\sec}_f \geq 0$ or $\leq 0$ as the hemisphere and hyperbolic space both admit densities with constant zero curvature.  It also shows that  $\overline{\sec}_f \geq 0$ or $\leq 0$ does not imply a triangle comparison theorem for the metric $h= e^{-2f}g$ since if $g$ is the hemisphere then $h$ is the hyperbolic space with the opposite curvature and vice-versa. 
  \end{remark}

\section{Conjugate Radius estimates}

In this section we discuss Jacobi field estimates. First we discuss the Cartan-Hadamard Theorem and then we  prove a theorem for a  positive  upper bound on weighted  curvature. 

\subsection{Weighted Cartan Hadamard theorem}

The Cartan-Hadamard theorem states that manifolds with non-positive sectional curvature do not have conjugate points.   First we show through an example that this theorem is not true for $\sec_X \leq 0$. 

\begin{example}  \label{cigar} Hamilton's cigar metric \cite{Hamilton} is a rotationally symmetric metric on $\mathbb{R}^2$ with $\mathrm{Ric} + \mathrm{Hess} f = 0$ and thus has $\sec_f = 0$.  However, it also has conjugate points. To see this note that  since the metric is simply connected and complete, if it had no conjugate points it would have a unique geodesic between any two points.  Since the cigar is rotationally symmetric we can write the metric as $g = dr^2 + \phi^2(r) d\theta^2$. In the coordinates $(r,\theta)$, fix $\theta_0$ and consider the geodesic defined for all $r \in (-\infty, \infty)$ 
\[ \gamma(r) = \left  \{ \begin{array}{ccc} (r, \theta_0) & r\geq 0 \\ (-r, -\theta_0) & r< 0  \end{array} \right. .\]
Then, since the cigar is cylindrical at infinity, there is a universal constant $C$ such that   $d(\gamma(r) ,\gamma(-r)) <C$, for all $r$.  In particular,   for $r > C/2$ there are  two geodesics between $\gamma(r)$ and $\gamma(-r)$, implying the metric has conjugate points.

\end{example}

On the other hand, we show that the stronger condition $\overline{\sec}_X \leq 0$ does imply the non-existence of  conjugate points. 

\begin{theorem} \label{CH_thm}
Suppose that a manifold $(M,g)$ supports a vector field such that  $\overline{\sec}_X \leq 0$, then $(M,g)$ has no conjugate points.
\end{theorem}

\begin{proof} Let $\gamma:[0, t_0] \rightarrow M$ be a unit speed geodesic and $J$ a Jacobi field along $\gamma$ which is perpendicular to $\gp$. Let $f=f_{\gamma}$ be the function $f_{\gamma}(t) = \int_0^t g_{\gamma(r)}(X, \gp(r)) dr$. 
Then we have 
\[ \frac{d}{dt} \left( \frac12 e^{-2f} |J|^2 \right) = e^{-2f_{\gamma}} \left( g\left( \dot{J} - g(X, \gp) J, J  \right) \right) \]
and 
\begin{eqnarray*}
  \frac{d}{dt} \left( g\left( \dot{J} - g(X, \gp) J, J  \right) \right) &=& g(\ddot{J}, J) + g(\dot{J}, \dot{J}) - \frac{1}{2} L_X g(\dot{\gamma}, \dot{\gamma})g(J,J) - 2g(X, \dot{\gamma})g(\dot{J},J)  \\
&=& -R(J, \dot{\gamma}, \dot{\gamma}, J) - \frac{1}{2} L_X g(\dot{\gamma}, \dot{\gamma})g(J,J) + g(\dot{J}, \dot{J}) - 2g(X, \dot{\gamma})g(\dot{J},J) \\
&=&  -R(J, \dot{\gamma}, \dot{\gamma}, J) - \frac{1}{2} L_X g(\dot{\gamma}, \dot{\gamma})g(J,J) +  |\dot{J} - g(X, \dot{\gamma})J|^2  - g(X, \dot{\gamma})^2g(J,J)  \\
&\geq&  |\dot{J} - g(X, \dot{\gamma})J|^2 - \overline{\sec}_X^{\gp}(J) |J|^2. 
\end{eqnarray*}

Then the assumption $\overline{\sec}_X \leq 0$ gives us that    $\frac{d}{dt}\left( g\left( \dot{J} - g(X, \gp) J, J  \right) \right) \geq 0$.   If  $J(0) = 0$, this implies that $g\left( \dot{J} - g(X, \gp) J, J  \right)  \geq 0,$ which gives us that  $\frac{d}{dt} \left( \frac12 e^{-2f} |J|^2 \right) \geq 0$.  Thus, the only way $J(0) = J(t_0)=0$ is if $J(t) = 0$ for all $0\leq t \leq t_0$.  
\end{proof}

\subsection{Positive Upper bound}

Now we consider the case $\overline{\sec}_X \leq K$, for a positive constant $K$. Recall that if  Riemannian manifold satisfies $\mathrm{sec} \leq K$  for some $K>0$ then  any two conjugate points are distance greater than or equal to $\frac{\pi}{\sqrt{K}}$ apart.  We generalize this result to the condition $\overline{\sec}_X \leq K$. 

To do so we fix some notation.  Given a fixed parametrized  geodesic $\gamma$ we let $u = e^{f_{\gamma}}$ and let $u_{max}$ and $u_{min}$ be the maximum and minimum of $u$ on the geodesic.  While the function $f_{\gamma}$ depends on the parametrization of $\gamma$ we note that the ratio $\frac{u_{min}}{u_{max}}$ does not.  

\begin{theorem} \label{Conjugate Radius}   If $\gamma$ is a geodesic such that $\overline{\sec}_X(\gp, E) \leq K$ for all $|E|=1$, $E \perp \gp$  then the distance between any two conjugate points of $\gamma$ is greater than or equal to $\frac{u_{\mathrm{min}}}{u_{\mathrm{max}}} \cdot \frac{\pi}{\sqrt{k}}.$
\end{theorem}

\begin{remark} We can obtain a different  proof of Theorem \ref{CH_thm}  by applying Theorem \ref{Conjugate Radius}  for $K \rightarrow 0$ for a fixed geodesic $\gamma$ with $\overline{\sec}_f(\dot{\gamma}, E) \leq 0$.  In particular, Theorem \ref{Conjugate Radius} is not true for  $\sec_X \leq K$.  \end{remark}

\begin{proof} [Proof of Theorem \ref{Conjugate Radius}]

Let  $J(t)$ a Jacobi field along $\gamma$ with $J(0) = 0$ and  let $\phi = \ln(\frac{1}{2} e^{-2f_{\gamma}} |J|^2)$. If $J(a) = 0$ then $\phi \rightarrow -\infty$  at $a$.    The derivative of $\phi$ is 
\begin{eqnarray*}
 \frac{d\phi}{dt}  &=& \frac{ 2 \left(g(\dot{J}, J)  - g(X, \gp)|J|^2 \right) }{ |J|^2} 
\end{eqnarray*}
Define $\lambda(t)  = \frac{1}{2} e^{2f_{\gamma}} \frac{d\phi}{dt}$.    Then 
%
%
\begin{eqnarray*}
\frac{d\lambda}{dt}  &=&  u^2 \left( \frac{ \frac{d}{dt} \left( g(\dot{J}, J) - g(X, \dot{\gamma})|J|^2 \right)  |J|^2 -2  \left( g(\dot{J}, J) - g(X, \dot{\gamma})\right)^2 } {e^{-2f_{\gamma}}|J|^4} \right)
\\  
&=& u^2 \left( \frac{  |\dot{J} - g(X, \gp)J|^2|J|^2  - 2  \left( g(\dot{J} - g(X, \dot{\gamma})J, J)  \right)^2 -  \overline{\sec}_X^{\gp}(J) |J|^4 }{|J|^4}\right)\\
&\geq &-\frac{\lambda^2}{u^2} - Ku^2 \\
\end{eqnarray*}
Where we have used the formula 
\begin{eqnarray*}
  \frac{d}{dt} \left( g\left( \dot{J} - g(X, \gp) J, J  \right) \right) &\geq&  |\dot{J} - g(X, \dot{\gamma})J|^2 - \overline{\sec}_X^{\gp}(J) |J|^2
  \end{eqnarray*}
 and Cauchy-Schwarz. 

We  thus have 
\[ \frac{d\lambda}{dt} \geq -\frac{\lambda^2}{u^2} - Ku^2 \geq -\frac{\lambda^2}{u_{\mathrm{min}} ^2} - Ku_{\mathrm{max}}^2.\]
We can then get a lower bound for $\lambda$ in terms of the solution to the corresponding Ricatti equation, $ \frac{d\lambda}{dt} =-\frac{\lambda^2}{u_{\mathrm{min}} ^2} - Ku_{\mathrm{max}}^2$. This equation  can be solved explicitly using separation of variables and we obtain

  \begin{eqnarray*}
   \lambda(t) &\geq& (u_{min} u_{max} \sqrt k ) \cot\left( \frac{u_{max} \sqrt{K}}{u_{min}} t \right) 
   \end{eqnarray*}
   This shows that $\lambda$ can not diverge to $-\infty$ for $t< \frac{u_{min}}{u_{max}} \frac{\pi}{\sqrt{K}}$, which implies  $J(t)$ can not go to $0$ for  $t< \frac{u_{min}}{u_{max}} \frac{\pi}{\sqrt{K}}$
\end{proof}

\section{Second variation of energy formula and Synge's theorem }

We now  discuss how the weighted curvatures appear in the formula for the second variation of energy of a path.  The energy of a path $c:[a,b] \rightarrow \mathbb{R}$  is 
\[ E(c) = \frac{1}{2}  \int_a^b |\dot{c}|^2 dt  \]
where $\dot{ }$ here and below will denote derivative in the $t$ direction.  The formula for the second variation of energy of geodesic is
\begin{eqnarray*}
\frac{d^2E}{ds^2} |_{s=0}
&=&  \int_{a}^b |\dot{V}|^2 - R(V, \gp, \gp, V) dt +  \left. g\left( \frac{\partial^2 \bar{\gamma}}{\partial s^2}, \dot{\gamma} \right) \right|_a^b \\
\end{eqnarray*}
where $\bar{\gamma}: [a, b] \times (-\varepsilon, \varepsilon) \rightarrow M$ is a variation of the geodesic $\gamma(t) = \bar{\gamma}(t, 0)$,  $V(t) = \frac{\partial \bar{\gamma}}{\partial s} |_{s=0}$ is the variation field.   The index form is the quantity 
\[ I_{[a,b]}(V,V) =   \int_{a}^b |\dot{V}|^2 - R(V, \gp, \gp, V) dt.\]
   Recall from section two  that weighted directional curvature operators along $\gamma$ are
\begin{eqnarray*}
R^{\gp } _X(U,V) &=&   R(U,\gp, \gp, V) +\frac12 L_Xg (\gp, \gp)g(U,V)\\
\overline{R}^{\gp } _X(U,V) &=&   R(U,\gp, \gp, V) +\frac12 L_Xg (\gp, \gp)g(U,V) + g(X,\gp)^2g(U,V).
\end{eqnarray*}
and that the weighted sectional curvatures are given by  $\sec^{\gp}_X(U) = R^{\gp } _X(U,U)$ and $\overline{\sec}^{\gp}_X(U) = \overline{R}^{\gp } _X(U,U)$  where $U$ is a unit vector perpendicular to $\gp$.   We can modify the formula for the index form to involve the weighted directional curvature operators. 
\begin{proposition} \label{2variation}
For the triple $(M,g,X)$ we have the following formulas for the Index form along a geodesic $\gamma$. 
\begin{eqnarray}
 I_{[a,b]}(V,V) &=& \int_{a}^b |\dot{V}|^2 - R^{\gp } _X(V,V)- 2 g(\dot{\gamma}, X) g(V, \dot{V}) dt + \left. g(\dot{\gamma}, X)|V|^2 \right|_a^b  \label{2vfa} \\
 &=&  \int_{a}^b |\dot{V} - g(\gp, X)V|^2  -\overline{R}^{\gp } _X(V,V) dt + \left. g(\dot{\gamma}, X)|V|^2 \right|_a^b  \label{2vfb}
\end{eqnarray}
 \end{proposition}
 
 \begin{proof} 
To obtain the first formula we write 
\begin{eqnarray*}
 I_{[a,b]}(V,V)&=& \int_{a}^b |\dot{V}|^2 -R^{\gp } _X(V,V) + \frac{1}{2}L_X(\dot{\gamma}, \dot{\gamma})|V|^2  dt  \\
&=&   \int_{a}^b |\dot{V}|^2 -R^{\gp } _X(V,V) + \left(\frac{d}{dt} g(\dot{\gamma}, X)\right) |V|^2  dt  \\
&=&  \int_{a}^b |\dot{V}|^2 -R^{\gp } _X(V,V) - g(\dot{\gamma}, X) \frac{d}{dt}  |V|^2 + \frac{d}{dt} \left(g(\dot{\gamma}, X)|V|^2 \right)   dt  \\
&=&   \int_{a}^b |\dot{V}|^2 -R^{\gp } _X(V,V) - 2 g(\dot{\gamma}, X) g(V, \dot{V}) dt + \left. g(\dot{\gamma}, X)|V|^2 \right|_a^b 
\end{eqnarray*}

To incorporate the strongly weighted curvature into the equation
we  complete the square
\[ |\dot{V} - g(\gp, X)V|^2 = |\dot{V}|^2 - 2g(\gp, X)g(V, \dot{V}) + g(\gp,X)^2|V|^2\] to obtain (\ref{2vfb}). 
\end{proof}

\begin{remark} When $X = \nabla f$, the weighted sectional curvatures $\sec_f^U(\gp)$ and $\overline{\sec}_f^U(\gp)$ also appear in the second variation formula for the weighted energy  at a weighted geodesic, see \cite{Morgan2, Morgan3}.  
\end{remark}

Our first application of these formulas will be to generalize Synge's theorem to the weighted setting.  We have the following lemma for parallel fields around closed geodesics. 

\begin{lemma}  \label{Shorter} Let  $(M,g,X)$ be a Riemannian manifold equipped with a smooth vector field $X$ which contains  a closed geodesic $\gamma$ which supports a unit parallel field perpendicular to $\gp$.  If  either  $\mathrm{sec}_X >0$, or   $X=\nabla f$ and $\overline{\mathrm{sec}}_f>0$,  then there is smooth closed curve which is homotopic to $\gamma$ and has shorter length. 
\end{lemma}

\begin{proof}
First consider the case $\mathrm{sec}_X >0$.  For a parallel field $V$ along a geodesic (\ref{2vfa}) implies 
\begin{eqnarray*}
\frac{d^2E}{ds^2} |_{s=0} &=&   -\int_{a}^bR^{\gp } _X(V,V) dt + \left. g(\dot{\gamma}, X) \right|_a^b +  \left. g\left( \frac{\partial^2 \bar{\gamma}}{\partial s^2}, \dot{\gamma} \right) \right|_a^b.
\end{eqnarray*}

If the geodesic is closed then the boundary terms cancel and  from $\mathrm{sec}_X >0$  we obtain  
\[ \frac{d^2E}{ds^2} |_{s=0} =   -\int_{a}^bR^{\gp } _X(V,V) dt < 0 \]
Which shows that the closed curve obtained from the variation has smaller length than the original closed geodesic. 

When $X=\nabla f$ and $\overline{\mathrm{sec}}_f>0$, let $Y=e^f V$, then
\[ \dot{Y}  = g(X, \gp)e^f V = g(X, \gp)Y \]
Applying (\ref{2vfb})  to the variation field  $Y$ we also get that the boundary terms cancel and we obtain
\[ \frac{d^2E}{ds^2} |_{s=0} =   -\int_{a}^b \overline{R}^{\gp}_X(Y, Y) dt < 0 \]
Again showing that there is a closed curve with smaller length. 
\end{proof}

The proof of  Synge's theorem now goes  exactly as in the classical case. 

\begin{theorem}[Synge's Theorem] Suppose that $M$ is a compact manifold supporting a vector field $X$ such that  either  $\mathrm{sec}_X >0$, or   $X=\nabla f$ and $\overline{\mathrm{sec}}_f>0$, then
 \begin{enumerate}
 \item If $M$ is even dimensional and orientable, then $M$ is simply connected. 
 \item If $M$ is odd-dimensional, then $M$ is orientable
 \end{enumerate}
 \end{theorem}
 
 \begin{proof}
The argument of Synge proceeds by contradiction and shows that if the topological conclusions do not hold then there is a closed geodesic with a parallel field which minimizes  length in its homotopy class, see e.g. Theorem 26 of \cite{Petersen}.   Applying Lemma \ref{Shorter} then  gives the desired contradiction in the weighted setting. 
 \end{proof}

\section{Diameter Estimate}

Now we prove the diameter estimate in the introduction.  We could give a proof of the result using the traditional second variation of energy argument and the formula from the previous section.  However, we will give a quicker  proof using the Bochner formula.  From formula (\ref{Ric1}) above we have 
\[   \partial_r  (\Delta_X r )= -|\mathrm{Hess} r|^2 - \mathrm{Ric}_X (\partial_r, \partial_r). \]

We can modify this equation to obtain an equation for $\mathrm{Ric}_X^{-(n-1)}$ in the following way:

\begin{lemma} \label{NegBochner}  Let $\gamma$ be a geodesic and let  $v = e^{\frac{f_{\gamma}}{n-1}}$.  Then 
\[ \partial_r (v^2 \Delta_X r) \leq  -v^2 \frac{(\Delta_X r)^2}{n-1} - v^2 \mathrm{Ric}_X^{-(n-1)}(\partial_r, \partial_r)\]
\end{lemma}
\begin{proof}
We have  
\begin{eqnarray*}
  \partial_r  (v^2 \Delta_X r ) &=& \left(  -|\mathrm{Hess} r|^2 - \mathrm{Ric}_X (\partial_r, \partial_r) + \frac{2 \partial_r f}{n-1} \Delta_X r \right) v^2 \\
  &\leq& \left( - \frac{(\Delta r)^2}{n-1}  + \frac{2 \partial_r f}{n-1} \Delta_X r - \mathrm{Ric}_X (\partial_r, \partial_r) \right)v^2 \\
  &=& \left(  - \frac{(\Delta_X r)^2}{n-1} -  \mathrm{Ric}^{-(n-1)}_X (\partial_r, \partial_r) \right)v^2 
  \end{eqnarray*}
  \end{proof}

This now gives us Theorem \ref{DiamIntro}. 

\begin{proof}[Proof of Theorem \ref{DiamIntro}]
Let $\gamma(r)$ be a minimizing unit speed geodesic  and let $\lambda(r) =  \frac{v^2 \Delta_X r}{n-1}$, then Lemma \ref{NegBochner} tells us that 
\[ \dot{\lambda} \leq -\frac{\lambda^2}{v^2} - k v^2 \] 
when $\mathrm{Ric}_X^{-(n-1)} \geq (n-1)  k$.  Let $v_{max}$ and $v_{min}$ be the minimum and maximum of $v$. .  Then we have 
\[ \dot{\lambda} \leq -\frac{\lambda^2}{v_{max}^2} - k v_{min}^2 \] 
This implies by the Ricatti comparison that 
\[ \lambda(t) \leq v_{\min} v_{\max} \sqrt{k} \cot\left( \frac{v_{\min}}{v_{\max}} \sqrt{k} t \right). \]
Since the right hand side goes to $-\infty$ at $t _0= \frac{v_{max}}{v_{min}} \cdot \frac{\pi}{\sqrt{k}}$,  $\lambda$ must go to $-\infty$ at some earlier time, meaning the geodesic will not be minimizing past $t_0$. 
\end{proof} 

\section{Pinching}
In this section we present the  proof of Theorem \ref{PinchingIntro}. 
What we will show is that the conjugate radius  and second variation estimates we already have combined with classical methods  give  a proof that any such manifold is  a homotopy sphere.   We will go into less detail in many of the arguments in this section and instead reference the textbooks \cite{doC, Petersen, Kling}. 

For submanifolds $A$ and $B$ in $M$ define the path space as  
 \[ \Omega_{A,B}(M) = \{ \gamma:[0,1] \rightarrow M, \gamma(0) = A, \gamma(1) = B \}
 \]  
 We consider the Energy $E:  \Omega_{A,B}(M) \rightarrow \mathbb{R}$ and  variation fields tangent to $A$ and $B$ at the end points. The  critical points are then  the geodesics perpendicular to $A$ and $B$ and  we say that the index of such a geodesic is $\geq k$ if there is a $k$-dimensional space of variation fields along the geodesic which have negative second variation. The first step in the proof is that  the diameter estimate in the previous section can be improved to an  index estimate in the case of a sectional curvature bound. 
 
 \begin{lemma} \label{lemmaindex}
 Suppose that $\overline{\sec}_f \geq k$, then if $\gamma$ is a geodesic of length longer than $ \frac{u_{\max}}{u_{\min}} \cdot \frac{\pi}{\sqrt{k}} $ than the index of $\gamma$ is greater than or equal to $(n-1)$. 
 \end{lemma}
 \begin{proof}
Along a geodesic $\gamma:[0,l] \rightarrow M$  with a proper variation, $V$,   the second variation formula (\ref{2vfb})  becomes 
\[   \frac{d^2E}{ds^2} |_{s=0} =   \int_{0}^l |\dot{V} - g(\gp, X)V|^2  - \overline{R}^{\gp } _X(V,V) dt \]    Choose  $E$ to be a unit length parallel field along $\gamma$ such that $E \perp \gp$, let  $\phi(t)$ be a function such that   $\phi(0) = 0$ and $\phi(l)=0$,  and let $V = \phi e^{f} E. $
Then we have
\[ \dot{V}  - g(\gp, X)V  = \dot{\phi} e^f E \]
Plugging $V$ into the second variation formula then gives 
\begin{eqnarray*}
  \frac{d^2E}{ds^2} |_{s=0}& =&    \int_{0}^le^{2f} \left( (\dot{\phi})^2  -  \phi^2\overline{R}^{\gp } _X(V,V)\right) dt \\
  &=& - \int_{0}^le^{2f} \left( \ddot{\phi} \phi + 2\dot{f}\dot{\phi} \phi   +  \phi^2 \overline{R}^{\gp } _X(V,V)\right) dt \\
  &\leq& - \int_{0}^le^{2f} \phi  \left( \ddot{\phi} + 2\dot{f}\dot{\phi}  +  k\phi  \right)dt 
 \end{eqnarray*}  Let $\psi$ be the solution to 
 \begin{eqnarray*} \label{BVE}  \ddot{\psi} + 2\dot{f}\dot{\psi}  + k\psi = 0 \qquad \psi(0) = 0 \qquad  \dot{\psi}(0) = 1. \end{eqnarray*}
 and let $L$ be the smallest positive number such that $\psi(L) = 0$.  If we show that $L \leq \frac{u_{\max}}{u_{\min}} \cdot \frac{\pi}{\sqrt{k}}$ then this will imply the result since we can then  construct $(n-1)$ linearly independent fields along $\gamma$ with $  \frac{d^2E}{ds^2} |_{s=0} \leq 0$ by taking the fields $V$ as above and defining $\phi = \psi$ on $[0,L]$ and $\phi(t) = 0$ for $t \in \left [L, \frac{u_{\max}}{u_{\min}} \cdot \frac{\pi}{\sqrt{k}}\right]$.  
 
 To see that  $L \leq \frac{u_{\max}}{u_{\min}} \cdot \frac{\pi}{\sqrt{k}}$, let $\lambda = \frac{e^{2f} \dot{\phi}}{\phi}$.  Then a simple calculation shows that 
  \[ \dot{\lambda} =  - \frac{\lambda^2}{u^2} -ku^2\]
The Ricatti comparison applied as in the proof of Theorem \ref{DiamIntro} then gives the result 
 
 \end{proof} 
 
 This index estimate  gives the following generalization of a sphere theorem of Berger  \cite{Berger1} which is  Theorem 33 in  \cite{Petersen}. 
 
 \begin{theorem} \label{BergerHomotopy}  If a compact Riemannian manifold has $\overline{\sec}_f  \geq k$ and 
 \[ \mathrm{inj}(M,g) \geq \frac{u_{\max}}{u_{\min}} \frac{\pi}{2\sqrt{k}} \]
 Then $M$ is a homotopy sphere. 
 \end{theorem}
 
 \begin{proof} 
 Under the hypothesis, every geodesic loop $\gamma$ such that $\gamma(0) = \gamma(l)= p$ must have length greater than or equal to $\frac{u_{max}}{u_{min}} \frac{\pi}{2\sqrt{k}}$. Then Lemma \ref{lemmaindex} implies that every geodesic in $\Omega_{p,p}$ has index greater than or equal to $(n-1)$.    This then implies that $M$ is $(n-1)$ connected and thus a homotopy sphere see Theorems 32 and 33 of \cite{Petersen} along with  Theorem 2.5.16 of \cite{Kling}. 
 \end{proof}
 
This shows that the key to proving a sphere theorem is to prove an injectivity radius estimate.     In the even dimensional case an injectivity radius estimate  follows from  Theorem \ref{Conjugate Radius} and Lemma \ref{Shorter}. 
\begin{theorem}
Suppose that $M$ is a compact  even dimensional simply connected manifold   such that  $0 \leq \overline{\sec}_f\leq L$  then  $\mathrm{inj}(M,g) \geq  \frac{u_{min}}{u_{max}} \frac{\pi}{\sqrt{L}}. $
\end{theorem}

\begin{proof}
Suppose that $\mathrm{inj}(M,g)< \frac{u_{min}}{u_{max}} \frac{\pi}{\sqrt{L}}$.  Then from Theorem \ref{Conjugate Radius} the  conjugate radius is larger than the injectivity radius.  This tells us that   there is a closed geodesic of length $\frac12\mathrm{inj}_M$.  From the proof of  Synge's  theorem, when the manifold is orientable and even-dimensional it is possible to construct a parallel field along the geodesic and from Lemma \ref{Shorter}  there is a variation which decreases the length of this closed curve.  However, it is possible to show that this leads to conjugate points of smaller distance apart, a contradiction, see the proof of  Theorem 30 of \cite{Petersen}. 
 \end{proof} 
 
The odd dimensional case is more difficult where the injectivity radius estimate is due to Klingenberg in the classical case.  However, from what we have already proved,  Klingenberg's  arguments carry over to the weighted setting.  First we have the homotopy lemma.

    \begin{lemma} [Klingenberg's homotopy lemma] Suppose that a Riemannian manifold $(M,g)$ has the property that no geodesic segment of length less than $\pi$ contains a conjugate point. Suppose that $p, q \in M$ such that $p$ and $q$ are joined by two distinct geodesics  $\gamma_0$ and $\gamma_1$ which are homotopic.  Then there exists a curve in the homotopy $\alpha_{t_0}$ such that 
    \[ \mathrm{length}(\alpha_{t_0}) \geq 2\pi - \mathrm{min}\{ \mathrm{length}(\gamma_i)\} \]
    \end{lemma}
    
\begin{proof} 
This is usually stated with the conjugate point estimate replaced with the condition $\mathrm{sec} \leq 1$.   However, as is pointed out in 2.6.5 of \cite{Kling}, the lemma holds with the same proof in this generality. 
\end{proof}

Now we can prove the injectivity radius estimate in all dimensions. 

\begin{theorem} \label{InjEst}
Suppose that $(M,g,f)$ is complete simply connected  and satisfies 
\[  \frac{1}{4}\left (\frac{u_{max}}{u_{min}} \right)^2 < L \leq \overline{\sec}_f \leq  \left(\frac{u_{min}}{u_{max}} \right)^2\] 
then $\mathrm{inj}(M,g)\geq \pi$
\end{theorem} 

\begin{proof} Since $\overline{\sec}_f \leq  \left(\frac{u_{\min}}{u_{\max}} \right)^2$ Theorem \ref{Conjugate Radius}   shows that the conjugate radius is less than or equal to $\pi$   so that  we can apply the homotopy lemma.  On the other hand, from  \ref{lemmaindex}, $ \overline{\sec}_f >   \frac{1}{4}\left (\frac{u_{max}}{u_{min}} \right)^2$  implies  that any geodesic of length longer than $\frac{\pi}{2}$ has index greater than or equal to $2$.  These are the only two elements about curvature used in the proof of the injectivity radius estimate of Klingenberg, see for example the proof of Proposition 3.1 on page 276  of \cite{doC}. 
\end{proof}

The proof of Theorem \ref{PinchingIntro} now follows as Theorem \ref{InjEst} and Theorem \ref{BergerHomotopy} showing the manifold is a homotopy sphere.

 \section{Killing Fields}
 
In this section we augment the previous considerations involving Jacobi fields and the second variation of energy formula by showing that the weighted sectional curvatures also come up naturally in formulas for Killing fields.   Recall that  for a Killing field   $V$  on a Riemannian manifold $(M,g)$ we have the following. 
 \begin{eqnarray}
\label{KF1} \frac12  \nabla \left( |V|^2\right) & =& - \nabla_V V \\
\label{KF2} \frac12  \mathrm{Hess}\left (|V|^2\right) (Y,Y) &=&  |\nabla_Y V|^2 - R(Y,V,V,Y)
 \end{eqnarray}
Now suppose we have a smooth manifold with smooth density  $(M,g, f)$ and consider the  function 
 \[ h= \frac{1}{2}e^{-2f} |V|^2. \] 
then we have the following formulas. 
 \begin{lemma} 
 Let $Y$ be a tangent vector, then 
 \begin{eqnarray*}
  \nabla h &=& -e^{-2f}\left( \nabla_V V + |V|^2 \nabla f \right)\\
 \mathrm{Hess} h(Y,Y)  &=&-  2df\otimes dh(Y,Y)+   |\nabla_Y (e^{-f} V)|^2\\
 && \quad  - e^{-2f} \left( R(V,Y,Y,V) + |V|^2 \mathrm{Hess} f(Y,Y) +  |V|^2df(Y)^2 \right)
 \end{eqnarray*}
 \end{lemma} 
 \begin{proof}
 
For the  first equation,  from the product rule we have  
 \begin{eqnarray*}
 dh = e^{-2f}\left(- |V|^2 df  + d\left( \frac12 |V|^2\right) \right)
 \end{eqnarray*}
 So that
 \begin{eqnarray*}
 dh(Y) = -e^{-2f}\left( g(\nabla_V V, Y) + g(\nabla f, Y)|V|^2 \right)
 \end{eqnarray*}
Differentiating this equation then gives us  
 \begin{eqnarray*} 
 \mathrm{Hess} h = e^{-2f}\left( 2|V|^2 df \otimes df -2  d\left( \frac12 |V|^2\right) \otimes df - 2 df \otimes d\left(  \frac{1}{2} |V|^2\right) - |V|^2 \mathrm{Hess} f + \mathrm{Hess} \left(\frac12 |V|^2\right) \right).
 \end{eqnarray*}
Plugging in (\ref{KF1}) and (\ref{KF2}) gives us 
 \begin{eqnarray*}
 \mathrm{Hess} h(Y,Y) &=&e^{-2f}  \left(  2 |V|^2 g(\nabla f, Y)^2 +  4 g(\nabla f, Y)g(\nabla_V V, Y) +   |\nabla_Y V|^2 \right.  \\
 && \qquad  \left. - R(Y,V,V,Y) -  |V|^2 \mathrm{Hess} f(Y,Y) \right) \\
 &=& |\nabla_Y (e^{-f} V)|^2 - e^{-2f}( R(Y,V,V,Y) +  |V|^2 \mathrm{Hess} f(Y,Y)) \\
 && + e^{-2f} (|V|^2g(\nabla f, Y)^2  + 2 g(\nabla f, Y)g(\nabla_V V, Y)).
 \end{eqnarray*}
 Then we also have
 \begin{eqnarray*}
  df\otimes dh(Y,Y) &=& -e^{-2f} g(\nabla f, Y) \left( g(\nabla_V V, Y) + g(\nabla f, Y)|V|^2 \right)\\
  &=& -e^{-2f}\left(  g(\nabla f, Y) g(\nabla_V V, Y) + g(\nabla f, Y)^2|V|^2\right)
  \end{eqnarray*}
  which tells us that 
  \[  2e^{-2f} g(\nabla f, Y)g(\nabla_V V, Y)  =  -  2df\otimes dh(Y,Y)  - 2e^{-2f}g(\nabla f, Y)^2|V|^2.  \]
  Plugging this in to the original gives   \begin{eqnarray*}
  &&\mathrm{Hess} h(Y,Y) +  2df\otimes dh(Y,Y) = \\
   &&  \quad |\nabla_Y (e^{-f} V)|^2 - e^{-2f} \left(  R(Y,V,V,Y) +  |V|^2 \mathrm{Hess} f(Y,Y) +  |V|^2g(\nabla f, Y)^2 \right)
  \end{eqnarray*}
  \end{proof}
  
  \begin{theorem} Suppose $(M,g)$ is a compact even dimensional manifold, if there is a function $f$ such that $ \overline{\mathrm{sec}}_f > 0$
  then every Killing field has a zero. 
  \end{theorem}
  
  
  \begin{proof} 
  The argument is by contradiction.  If there is a vector field $V$ which does not have a zero then the function  $h$ has a non-zero minimum at a point  $p$.   At $p$,   we then have $dh=0$ which implies from the previous lemma that 
  \[ g(\nabla_V V, Y) = - g(\nabla f, Y)|V|^2 \qquad \forall Y \in T_pM \]
    In particular,   setting $Y=V$ and using the skew-symmetry of $\nabla V$ we obtain  $g(\nabla f, V) = 0$ at $p$. 
    
    Consider the skew symmetric endomorphism on $A : T_pM \rightarrow T_pM $ given by 
   \[ A(w) = \nabla_w V  + g(w, V) \nabla f - g(w, \nabla f) V \]
   Then,  using that $V \perp \nabla f$ at $p$ we can see that $V|_p$ is in the null space of $A$ as 
   \[ A(V|_p) = (\nabla_V V + |V|^2 \nabla f)|_p = \nabla h|_p  = 0 \]
   If the dimension of the manifold is even, then we know that $A$ has another zero eigenvector for $A$ which is perpendicular to $V$, call it $w$.  Then we have  
   \[ 0 = A(w) = \nabla_w V  - g(w, \nabla f) V \]
   Which implies that  
   \[ \nabla_w(e^{-f} V) = e^{-f} A(w) = 0 \]
  Plugging this into the formula for the Hessian of $h$ in the previous lemma gives us 
  \begin{eqnarray*}
  \mathrm{Hess} h(w,w)&=& - e^{-2f} \left(  R(w,V,V,w) +  |V|^2 \mathrm{Hess} f(w,w) +  |V|^2g(\nabla f, w)^2  \right) 
  \end{eqnarray*}
 The assumption  $\overline{\mathrm{sec}}_f > 0$ then shows that $\mathrm{Hess} h (w,w)<0$, which is a contradiction to $p$ being a minimum. 
\end{proof}

\end{document}